\documentclass[11pt]{amsart}
\usepackage{comment}
\usepackage{amssymb,amsmath,amsthm,mathrsfs,amstext,amssymb}
\usepackage{tabularx}
\usepackage{array}
\usepackage[square,sort,comma,numbers]{natbib}
\usepackage{url} 
\usepackage{hyperref}
\usepackage{accents}

\newtheorem{theorem}{Theorem}
\numberwithin{theorem}{section}
\newtheorem*{theorem*}{Theorem}
\newtheorem{lemma}[theorem]{Lemma}
\newtheorem*{lemma*}{Lemma}

\newtheorem*{corollary*}{Corollary}
\newtheorem{proposition}[theorem]{Proposition}
\newtheorem*{proposition*}{Proposition}

\theoremstyle{definition}

\newtheorem{definition}[theorem]{Definition}
\newtheorem*{definition*}{Definition}
\newtheorem{remark}[theorem]{Remark}
\newtheorem*{remark*}{Remark}
\newtheorem{question}[theorem]{Question}
\newtheorem*{question*}{Question}

\newcommand{\A}{{\mathcal{A}}}

\newcommand{\C}{{\mathcal{C}}}

\newcommand{\F}{{\mathcal{F}}}

\newcommand{\I}{{\mathcal{I}}}

\newcommand{\M}{{\mathcal{M}}}

\renewcommand{\O}{{\mathcal{O}}}
\renewcommand{\P}{{\mathcal{P}}}

\renewcommand{\S}{{\mathcal{S}}}
\newcommand{\T}{{\mathcal{T}}}
\newcommand{\U}{{\mathcal{U}}}

\newcommand{\W}{{\mathcal{W}}}

\newcommand{\CC}{{\mathbb{C}}}

\newcommand{\PP}{{\mathbb{P}}}
\newcommand{\QQ}{{\mathbb{Q}}}

\renewcommand{\SS}{{\mathbb{S}}}
\newcommand{\TT}{{\mathbb{T}}}

\renewcommand{\a}{{\mathfrak{a}}}
\renewcommand{\b}{{\mathfrak{b}}}
\renewcommand{\c}{{\mathfrak{c}}}
\renewcommand{\d}{{\mathfrak{d}}}

\renewcommand{\i}{{\mathfrak{i}}}

\renewcommand{\u}{{\mathfrak{u}}}

\DeclareMathOperator{\dom}{dom}
\DeclareMathOperator{\ran}{ran}
\DeclareMathOperator{\restr}{\upharpoonright}
\DeclareMathOperator{\concat}{{^\smallfrown}}

\newcommand{\simpleset}[1]{{\{{#1}\}}}
\newcommand{\simpleseq}[1]{{\langle{#1}\rangle}}
\newcommand{\set}[2]{{\{ {#1} \mid {#2} \}}}
\newcommand{\seq}[2]{{\langle {#1} \mid {#2} \rangle}}

\DeclareMathOperator{\completesubposet}{{\mathbin{< \!\!\! \circ}}}
\DeclareMathOperator{\extends}{{\mathbin{\leq}}}

\DeclareMathOperator{\compat}{{\mathbin{||}}}

\DeclareMathOperator{\forces}{{ \, \Vdash \, }}

\DeclareMathOperator{\forcingconcat}{{\mathbin{\ast}}}
\newcommand{\gen}{{\dot{G}}}

\newcommand{\finseq}[1]{{{#1}^{<\omega}}}
\newcommand{\infseq}[1]{{{#1}^{\omega}}}
\newcommand{\finsubset}[1]{{[#1]^{<\omega}}}

\DeclareMathOperator{\finseqomega}{\finseq{\omega}}
\DeclareMathOperator{\infseqomega}{\infseq{\omega}}

\DeclareMathOperator{\cov}{cov}
\DeclareMathOperator{\cof}{cof}

\DeclareMathOperator{\spec}{spec}

\DeclareMathOperator{\aT}{\a_{T}}

\DeclareMathOperator{\suc}{succ}
\DeclareMathOperator{\sucspl}{succspl}

\DeclareMathOperator{\spl}{spl}
\DeclareMathOperator{\stem}{stem}

\newcommand{\finseqomegaarg}[1]{{{\omega}^{<{#1}}}}
\newcommand{\finseqeqomegaarg}[1]{{{\omega}^{{#1}}}}
\newcommand{\finseqleqomegaarg}[1]{{{\omega}^{\leq{#1}}}}

\title{Partitions of the Baire space into compact sets}

\author{V. Fischer}
\address{Institute of Mathematics, University of Vienna, Kolingasse 14-16, 1090 Vienna, Austria}
\email{vera.fischer@univie.ac.at}

\author{L. Schembecker}
\address{Department of Mathematics, University of Hamburg, Bundesstrasse 55, 20146 Hamburg, Germany}
\email{lukas.schembecker@uni-hamburg.de}

\thanks{\emph{Acknowledgments.}: The authors would like to thank the Austrian Science Fund (FWF) for the generous support through START Grant Y1012-N35.}



\begin{document}
	\maketitle
	
	\begin{abstract}
		We define a c.c.c.\ forcing which adds a maximal almost disjoint family of finitely splitting trees on $\omega$ (a.d.f.s.\ family) or equivalently a partition of the Baire space into compact sets of desired size.
		Further, we utilize the forcing to add arbitrarily many maximal a.d.f.s.\ families of arbitrary sizes at the same time, so that the spectrum of $\aT$ may be large.
		
		Furthermore, under CH we construct a Sacks-indestructible maximal a.d.f.s.\ family (by countably supported iteration and product of any length), which answers a question of Newelski~\cite{New87}. 
		Also, we present an in-depth ``isomorphism of names"-argument to show that in generic extensions of models of CH by countably supported Sacks forcing there are no maximal a.d.f.s.\ families of size $\kappa$, where $\aleph_1 < \kappa < \c$.
		Thus, we prove that in the generic extension the spectrum of $\aT$ is $\simpleset{\aleph_1, \c}$. Finally, we prove that Shelah's ultrapower model of \cite{shelahtemplate} for the consistency of $\d < \a$ also satisfies $\a = \aT$.
		Thus, consistently $\d < \a = \aT$ may hold.
	\end{abstract}

\section{Introduction}

Given an uncountable Polish space $X$ and a pointclass $\Gamma$ of Borel sets we want to understand what are the possible cardinalities of  partitions of $X$ into sets in $\Gamma$. Here, with a partition we mean a collection of non-empty subsets of $X$ which are pairwise disjoint and union up to the whole space $X$.

We begin with a brief recollection of when partitions of size $\aleph_1$ exist for different pointclasses $\Gamma$. For $\Gamma = F_{\sigma \delta}$, Hausdorff proved in \cite{hausdorff} that every Polish space is the union of $\aleph_1$-many strictly increasing $G_\delta$-sets. Thus, one immediately obtains that every uncountable Polish space can be partitioned into $\aleph_1$-many disjoint $F_{\sigma\delta}$-sets. For $\Gamma = G_\delta$, there is a close connection to $\cov(\M)$, the minimal size of a family of meager sets which covers $^\omega 2$. In fact, in \cite{fremlinshelah} Fremlin and Shelah showed that $\cov(\M) = \aleph_1$ if and only if $^\omega 2$ can be partitioned into $\aleph_1$-many $G_\delta$-sets. Finally, the most interesting case for us will be partitions into closed/compact sets. In \cite{miller} Miller introduced a proper forcing notion, which is $\infseqomega$-bounding (see~\cite{spinas}) and destroys a given partition $\C$ of $^\omega 2$ into closed sets by adding a new real which is not in the closure of any element of $\C$ in the generic extension. Thus, by iterating the forcing $\aleph_2$-many times and using a suitable bookkeeping one obtains a model in which there is no partition of $^\omega 2$ into $\aleph_1$-many closed sets.
We denote with $\aT$ the minimal size of a partition of $^\omega 2$ into closed sets, so with this notation in~\cite{miller} Miller established the relative consistency of $\cov(\M) = \aleph_1 < \aT = \aleph_2$. 
Notably, the forcing also preserves tight mad families (see~ \cite{restricted}), and as more recently shown in~\cite{partitionforcing} selective independent families and P-points. Thus, the same model witnesses the consistency of $\d = \a = \i = \u = \aleph_1 < \aT = \aleph_2$ (see~\cite{partitionforcing}).

Naturally, one might ask if the definition of $\aT$ differs if we would have chosen another uncountable Polish space than $^\omega 2$.  However, by a recent result of Brian~\cite{brian}, some uncountable Polish space can be partitioned into $\kappa$-many closed sets if and only if every uncountable Polish space can be partitioned into $\kappa$-many closed sets. Hence, not only $\aT$ is independent of the choice of the underlying Polish space, but so is also  $\spec(\aT)$, which is the set of all sizes of partitions of $^\omega 2$ into closed sets.
In fact, even more is true~\cite[Theorem 2.4]{brian}: The existence of partitions into $\kappa$-many closed sets is equivalent to the existence of partitions into $\kappa$-many compact or $F_\sigma$-sets. Thus, if we would be only interested in the possible values of $\aT$ and $\spec(\aT)$ it does not matter which uncountable Polish space we partition or whether these partitions are into compact, closed or $F_\sigma$-sets. In addition, Brian established in \cite{brian} the existence of a c.c.c.\ forcing that adds partitions of $^\omega 2$ into $F_\sigma$-sets of desired sizes and only those sizes. Therefore by the aforementioned theorem the poset adjoins implicitly also partitions of $^\omega 2$ into closed or compact sets of these desired sizes.

In contrast, in the current paper we introduce a c.c.c.\ forcing that explicitly adds a partition of $\infseqomega$ into compact sets. Our approach stems from a slightly different standpoint, as $\aT$ can also be defined as the minimal size of a maximal almost disjoint family of finitely splitting trees on $\finseqomega$, since König's Lemma implies that such families can be identified with partitions of $\infseqomega$ into compact sets (for a more detailed discussion see Section~\ref{prelim}). Hence, it will be most natural for us to consider partitions of Baire space into compact sets.

This paper is structured as follows:
In section 2 we will review some common notions and definitions regarding trees.
The third section begins with the introduction of a c.c.c.\ forcing which extends a given almost disjoint family of finitely splitting trees by $\omega$-many new finitely splitting trees (see Definition \ref{forcingdef}), so that the extended family is still almost disjoint (see Lemma \ref{generic_adfs}).
We then prove that the generic new trees satisfy a certain diagonalization property (see Definition \ref{diagdef1} and \ref{diagdef2} and Proposition \ref{diag}), so that iterating the forcing yields a maximal almost disjoint family of finitely splitting trees of desired size $\kappa$ for any $\kappa$ of uncountable cofinality (see Theorem \ref{key}).
We immediately obtain the following consistency result, see Theorem \ref{thm.aT.c.arbitrary}

\begin{theorem*}
	Assume CH and let $\kappa \leq \lambda$ be regular cardinals.
	Then there is c.c.c.\ extension in which  $\d = \aT = \kappa \leq \lambda = \c$ holds.
\end{theorem*}

Further, we utilize the forcing to add arbitrarily many maximal a.d.f.s.\ families of arbitrary sizes at the same time (see Theorem \ref{manyadfs}):

\begin{theorem*}
	Assume GCH.
	Let $\lambda$ be a cardinal of uncountable cofinality, $\theta \leq \lambda$ a regular cardinal and $\seq{\kappa_\beta}{\beta < \theta}$ a sequence of regular uncountable cardinals with $\cof(\lambda) \leq \kappa_\beta \leq \lambda$ for all $\beta < \theta$.
	Then there is a c.c.c.\ extension in which $\c = \lambda$ and $\kappa_\beta \in \spec(\aT)$ for all $\beta < \theta$.
\end{theorem*}

Finally, we conclude section 3 with a brief analysis if and when the forcing adds unbounded, 
Cohen and dominating reals (see Remark \ref{unbounded} and Propositions \ref{cohen} and \ref{dominating}).

Continuing the recent results on the existence of Sacks-indestructible witnesses for various combinatorial properties, e.g. independent families \cite{sacksmif1}, \cite{sacksmif2} or eventually different families \cite{sacksmed}, in the fourth section we extend this analysis to almost disjoint families of finitely splitting trees (see Theorem \ref{indestructible_adfs}):

\begin{theorem*}[CH]
	There is an a.d.f.s.\ family which remains maximal after forcing with countably supported iteration or product of Sacks forcing of arbitrary length.
\end{theorem*}

The above result answers in particular Question 2 of Newelski's~\cite{New87}, see the discussion at the end of section 4. First, we give a brief reminder of common definitions of Sacks forcing and fusion for countably supported products of Sacks forcing.
For the proof of Theorem \ref{indestructible_adfs} we will do an analogous construction as for the construction of a Sacks-indestructible maximal eventually different family by Fischer and Schrittesser in \cite{sacksmed}.

Thus, the two main ingredients of Theorem \ref{indestructible_adfs} are a particularly nice version of continuous reading of names for countably supported iterations of Sacks forcing (see Lemma 4 of \cite{sacksmed}) and the following Lemma \ref{extendctbladfs}:

\begin{lemma*}
Let $\T$ be a countable a.d.f.s.\ family, $\lambda$ be a cardinal, $p \in \SS^\lambda$ and $\dot{f}$ be a $\SS^\lambda$-name for a real such that for all $T \in \T$ we have
$$p \forces \dot{f} \notin [T].$$
Then there is a finitely splitting tree $S$ and $q \extends p$ such that $\T \cup \simpleset{S}$ is an a.d.f.s.\ family and
$$q \forces \dot{f} \in [S].$$
\end{lemma*}

Here, $\SS^\lambda$ is the countably supported product of Sacks forcing of size $\lambda$.
Under CH we may apply this lemma in a diagonal fashion to ensure that every new real may not be a witness that our a.d.f.s.\ family is not maximal in the generic extension (see Theorem \ref{indestructible_adfs}).
Notice the similarity to Lemma 7 in \cite{sacksmed} which can be rephrased as

\begin{lemma*}
	Let $\F$ be a countable eventually different (e.d.) family, $\lambda$ be a cardinal, $p \in \SS^\lambda$ and $\dot{f}$ be a $\SS^\lambda$-name for a real such that for all $g \in \F$ we have
	$$p \forces \dot{f} \text{ is eventually different from } g.$$
	Then there is a real $h$ and $q \extends p$ such that $\F \cup \simpleset{h}$ is an e.d.\ family and
	$$q \forces \dot{f} \text{ is not eventually different from } h.$$
\end{lemma*}

We conclude section 4 with an in-depth ``isomorphism of names"-argument to show that in generic extensions of models of CH by countably supported Sacks forcing there are no maximal a.d.f.s.\ families of size $\kappa$, where $\aleph_1 < \kappa < \c$.
Thus, by combining both results we obtain Theorem~\ref{thm.small.spectra}:

\begin{theorem*}[CH] \label{specsacks}
	Let $\lambda$ be a cardinal such that $\lambda^{\aleph_0} = \lambda$. Then $\SS^\lambda \forces ``\spec(\aT) = \simpleset{\aleph_1, \lambda}$".
\end{theorem*}

Finally, in the last section we will use an ``average of names"-argument to prove that in the template model constructed by Shelah in \cite{shelahtemplate} to obtain the consistency of $\d < \a$, we also have that $\a = \aT$.
Thus, we extend Shelah's theorem to the following (see Theorem \ref{thm.aT.ultrapowers}):

\begin{theorem*}
	Assume $\kappa$ is measurable, and $\kappa < \mu < \lambda$, $\lambda = \lambda^\omega$, $\nu^\kappa < \lambda$ for all $\nu < \lambda$, are regular cardinals.
	Then there is a forcing extension satisfying $\b = \d = \mu$ and $\a = \aT = \c = \lambda$.
\end{theorem*}

\section{Preliminaries}\label{prelim}

In the following every tree $T$ will be a tree on $\omega$, i.e.\ $T \subseteq \finseqomega$ is non-empty and closed under initial sequences.
For a tree $T$, we recall the following notions:

\begin{enumerate}
	\item If $s \in T$ and $n < \omega$, then $\suc_T(s) = \set{n < \omega}{s \concat n \in T}$ and $T \restr n = T \cap \finseqleqomegaarg{n}$.
	\item $T$ is pruned iff $\suc_T(s) \neq \emptyset$ for all $s \in T$.
	\item $T$ is finitely splitting iff $T$ is pruned and $\suc_T(s)$ is finite for all $s \in T$.
	\item $[T] = \set{f \in \infseqomega}{(\forall n < \omega)(f \restr n \in T)}$.
\end{enumerate}
For $n < \omega$ we call $T \subseteq \finseqleqomegaarg{n}$ an $n$-tree and use the same definitions as above, where we replace $(2)\&(4)$ with $(2^*)\&(4^*)$, respectively:
\begin{enumerate}
	\item[(2*)] $T$ is pruned iff $\suc_T(s) \neq \emptyset$ for all $s \in T \cap \finseqomegaarg{n}$.
	\item[(4*)] $[T] = T \cap \finseqeqomegaarg{n}$.
\end{enumerate}
For trees $T \subseteq \finseqomega$ we call $[T]$ the branches of $T$ and for $n$-trees $T \subseteq \finseqleqomegaarg{n}$ we call $[T]$ the leaves of $T$.
Remember that the non-empty closed sets of $\finseqomega$ are in bijection with pruned trees on $\omega$ in the following way.
Let $T \subseteq \finseqomega$ be a pruned tree and $C \subseteq \infseqomega$ be closed and non-empty, then the maps
$$
T \mapsto [T] \quad \text{and} \quad C \mapsto T_C = \set{s \in \finseqomega}{(\exists f \in C)(s \subseteq f)}
$$
are inverse to each other.
Furthermore, the bijection restricts to a bijection between finitely splitting trees on $\omega$ and non-empty compact subsets of $\infseqomega$. The following object will be of central interest for our study:
\begin{definition}\label{spec_aT}
The partition spectrum of the Baire space is the set 
$$\spec(\aT) = \set{|P|}{P \text{ is a partition of } \infseqomega \text{ into compact sets}}$$ and $\aT = \min\spec(\aT)$.	
\end{definition}

As $\infseqomega$ is not $\sigma$-compact, $\aleph_1 \leq \aT$ and the partition of $\infseqomega$ into singletons witnesses that $\aT \leq \c$.
Furthermore, by the aforementioned result of Brian~\cite{brian}, $\aT$ and $\spec(\aT)$ do not depend on the choice of the underlying Polish space $X$ and also do not depend on the choice of compact, closed or $F_\sigma$ partitions.

\begin{definition}
	A family $\T$ of finitely splitting trees is called an almost disjoint family of finitely splitting trees (or an a.d.f.s.\ family) iff $S$ and $T$ are almost disjoint, i.e.\ $S \cap T$ is finite for all $S,T \in \T$.
	It is called maximal iff there is no finitely splitting tree $S$ which is almost disjoint from every $T \in \T$.
\end{definition}

Almost disjoint families of finitely splitting trees will be the combinatorial object of main interest for this paper.	
Notice that by König's lemma for finitely splitting trees $S$ and $T$ we have that $S$ and $T$ are almost disjoint iff $[T] \cap [S] = \emptyset$.
Thus, using the above identification of finitely splitting trees and non-empty compact subsets of $\infseqomega$, we can also identify maximal a.d.f.s.\ families with partitions of $\infseqomega$ into non-empty compact sets. Moreover, we get that an a.d.f.s.\ family $\T$ is maximal iff for all reals $f \in \infseqomega$ there is a $T \in \T$ such that $f \in [T]$. To conclude this section we note that $\d \leq \aT$ \cite{spinas} and $\spec(\aT)$ is closed under singular limits \cite{brian}.

\section{Forcing the existence of maximal a.d.f.s.\ families}

In this chapter we will define and analyse a c.c.c.\ forcing that allows us for any $\kappa$ of uncountable cofinality to explicitly add a maximal a.d.f.s.\ family of size $\kappa$ or equivalently a partition of $\infseqomega$ into $\kappa$-many compact sets. We will also see that we can use this forcing to change the value of $\aT$ to any regular uncountable cardinal. These results mainly depend on the fact that our forcing diagonalizes a given a.d.f.s\. family in the following sense:

\begin{definition} \label{diagdef1}
Let $\T$ be an a.d.f.s.\ family.	We define
$$\W(\T) = \set{f \in \infseqomega}{\text{for all } T \in \T \text{ we have  } f \notin [T]}.$$ Furthermore, we define
$$\I^+(\T) = \set{T}{T \text{ is a finitely splitting tree with } [T] \subseteq \W(\T)}.$$ 
\end{definition}

Note that $\W(\T)$ is the set of all reals that $\T$ is missing to be maximal and $S \in \I^+(\T)$ iff $S$ is almost disjoint from every $T \in \T$.
In particular, $\I^+(\T)$ consists of exactly those finitely splitting trees that we may add to $\T$ without destroying almost disjointness. The notion $\I^+(\T)$ should emphasize the similarity to $\I^+(\A)$ for an a.d. family $\A$ in the subsequent diagonalization property.
However, contrary to $\I^+(\A)$ it is generally not induced by an ideal.

\begin{definition} \label{diagdef2}
Let $\T$ be an a.d.f.s.\ family, let $\PP$ be a forcing notion and $G$ be a $\PP$-generic filter. We say that a set of finitely splitting trees $\S$ in $V[G]$ diagonalizes $\T$ iff $\T \cup \S$ is an a.d.f.s.\ family and for all $T \in \I^+(\T)^V$ we have that $\simpleset{T} \cup \S$ is not almost disjoint.
\end{definition}

\begin{lemma}
	The above definition is equivalent to: $\T \cup \S$ is an a.d.f.s.\ family and for all $f \in \W(\T)^V$ there is an $S \in \S$ with $f \in [S]$.
\end{lemma}

\begin{proof}
We argue in $V[G]$. Let $f \in \W(\T)^V$. Then $T_f = \set{s \in \finseqomega}{s \subseteq f} \in \I^+(\T)^V$. By assumption choose $S \in \S$ such that $[T_f] \cap [S] \neq \emptyset$. But $[T_f] = \simpleset{f}$, so $f \in [S]$.
	
Conversely, let $T \in \I^+(\T)^V$. Choose $f \in \W(\T)^V$ such that $f \in [T]$. By assumption choose $S \in \S$ such that $f \in [S]$. But then $f \in [S] \cap [T]$, so $S$ and $T$ are not almost disjoint.
\end{proof}

The above observation motivates the following definition:

\begin{definition}
$\TT$ is the forcing consisting of partial functions $p:\omega \times \finseqomega \to 2$, such that 
\begin{enumerate}
\item $\dom(p) = F_p \times \finseqleqomegaarg{n_p}$, where $n_p \in \omega$ and $F_p \in \finsubset{\omega}$;
\item for all $i \in F_p$, $T^i_p = \set{s \in \finseqleqomegaarg{n_p}}{p(i)(s) = 1} (\subseteq \finseqleqomegaarg{n_p})$ is a finitely splitting $n_p$-tree.
\end{enumerate}
We order $\TT$ by $p \extends q$ iff $p \supseteq q$.
Note that this implies that $T_p^i$ end-extends $T_q^i$ for every $i \in F_q$.
\end{definition}

\begin{definition}\label{forcingdef}
Let $\T$ be an a.d.f.s.\ family. $\TT(\T)$ is the forcing notion of all pairs $(p, w_p)$ such that $p \in \TT$ and $w_p: \W(\T) \to \omega$ is a partial function, so that $\dom(w_p) = H_p$ for some $H_p \in \finsubset{\W(\T)}$ and for all $f \in H_p$ we have $w_p(f) \notin F_p$ or $f \restr n_p \in T^{w_p(f)}_p$.
We order $\TT(\T)$ by
	\begin{align*}
	(p, w_p) \extends (q, w_q) \quad \text{iff} \quad &p \extends_{\TT} q \text{ and } w_p \supseteq w_q
	\end{align*}
\end{definition}

Notice that for every maximal a.d.f.s.\ family $\T$ we have $\W(\T) = \emptyset$, so $\TT(\T) \cong \TT \cong \CC$ as $\TT$ is countable.
The intuition is that $\TT(\T)$ adds $\omega$-many new finitely splitting trees, where the side conditions $w_p$ will ensure that every element of $\W(\T)$ will be contained in the branches of one those new trees.
The main difference with the forcing sketched in~\cite[Theorem 3.1]{brian} is that only one of the new trees is allowed to contain an element of $\W(\T)$ as its branch, which ensures that the new trees are almost disjoint. Thus, this forcing may be used to explicitly add partitions of $\infseqomega$ into compact sets rather than only a partition into $F_\sigma$-sets.

\begin{lemma}\label{Knaster}
	$\TT(\T)$ is Knaster.
\end{lemma}

\begin{proof}
Let $A \subseteq \TT(\T)$ be of size $\omega_1$.
Since	$\TT$ is countable, we may assume that $p = q$ for all $(p, w_p),(q, w_q) \in A$. Moreover, by the $\Delta$-system lemma applied to $\set{H_p}{(p, w_p) \in A}$, we may assume that there is a root $R \in \finsubset{\W(\T)}$, i.e.\ that $H_p \cap H_q = R$ for all $(p, w_p) \neq (q, w_q) \in A$. However, there are only countably many functions from $R \to \omega$, so we may assume $w_p \restr R = w_q \restr R$ holds for all $(p, w_p),(q, w_q) \in A$.
	
It remains to observe, that then, the elements of $A$ are pairwise compatible. Indeed, consider $(p, w_p), (q, w_q) \in A$. By choice of $R$, $w_{p \cup q} = w_p \cup w_q$ is a function with $\dom(w_{p \cup q}) = H_p \cup H_q$. We claim that $(p, w_{p \cup q}) \in \TT(\T)$. If this is the case, then the result follows as $p = q = p \cup q$ implies that $(p \cup q, w_{p \cup q}) \extends (p, w_p), (q, w_q)$.
	
Let $f \in H_p \cup H_q$. If $f \in H_p$, then either $w_p(f) \in F_p$ and  so $f \restr n \in T^{w_p(f)}_p = T^{w_{p \cup q}(f)}_{p \cup q}$, or otherwise we have $w_p(f) \notin F_p$ and so  $w_p(f) \notin F_q$,  $w_p(f) \notin F_{p \cup q}$. Analogously, if $f \in H_q$, then either $w_q(f) \in F_q$ and so $f \restr n \in T^{w_q(f)}_q = T^{w_{p \cup q}(f)}_{p \cup q}$, or otherwise we have $w_q(f) \notin F_q$ and so $w_q(f) \notin F_p$, $w_q(f) \notin F_{p \cup q}$.
\end{proof}

\begin{definition}
Let $\T$ be an a.d.f.s.\ family and let $G$ be a $\TT(\T)$-generic filter. In $V[G]$ we let $\S_G = \set{S_{G,i}}{i \in \omega}$, where $S_{G,i} = \set{s \in \finseqomega}{ \text{there is } (p, w_p) \in G \text{ with } p(i)(s) = 1}$ for $i \in \omega$.
\end{definition}

Next, we show that in $V[G]$ the family $\S_G$ diagonalizes $\T$.
First, we prove that $\T \cup \S$ is an a.d.f.s.\ family.
We will make use of the following density arguments:

\begin{proposition}\label{prop2.8}
Let $\T$ be an a.d.f.s.\ family. Let $(p, w_p) \in \TT(\T)$ and $i_0 \in \omega \setminus F_p$. Then there is  $q\in\mathbb{T}$ such that $q\leq_{\mathbb{T}} p$,  $\dom(q) = (F_p \cup \simpleset{i_0}) \times \finseqleqomegaarg{n_p}$, $(q, w_p) \in \TT(\T)$ and $(q, w_p) \extends (p, w_p)$.
\end{proposition}

\begin{proof}
Choose any finitely splitting $n_p$-tree $T$ which contains the leaf $f \restr n_p$ for every $f \in H_p$ with $w_p(f) = i_0$ and define $q:(F_p \cup \simpleset{i_0}) \times \finseqleqomegaarg{n_p}$ by
	$$
	T^i_q =
	\begin{cases}
	T		& \text{if } i = i_0\\
	T^i_p   & \text{otherwise}
	\end{cases}
	$$
	Then $q \in \TT$ and $q \extends p$.
	Furthermore, we have that $(q, w_p) \in \TT(\T)$ by choice of $T$ and clearly also $(q, w_p) \extends (p, w_p)$ holds by definition of $q$.
\end{proof}

\pagebreak

\begin{proposition}\label{prop2.9}
	Let $\T$ be an a.d.f.s.\ family.
	Let $(p, w_p) \in \TT(\T)$ and $m > n_p$.
	Then we can extend $p$ to $q$ with $\dom(q) = F_p \times \finseqleqomegaarg{m}$, $(q, w_p) \in \TT(\T)$ and $(q, w_p) \extends (p, w_p)$.
\end{proposition}

\begin{proof}
	For every $i \in F_p$ and for every $f \in H_p$ with $w_p(f) = i$ we have $f \restr n_p \in T^i_p$.
	Hence, we may choose a finitely splitting $m$-tree $T_i$ which extends $T^i_p$ and contains $f \restr m$ for every $f \in H_p$ with $w_p(f) = i$.
	Define $q:F_p \times \finseqleqomegaarg{m}$ by $T^i_q = T_i$.
	Then $q \in \TT$ and $q \extends p$.
	Furthermore, we have that $(q, w_p) \in \TT(\T)$ by choice of the $T_i$'s and clearly also $(q, w_p) \extends (p, w_p)$ holds by definition of $q$.
\end{proof}

\begin{proposition}\label{prop2.10}
Let $\T = \set{T_j}{j \in J}$ be an a.d.f.s.\ family. Let $(p, w_p) \in \TT(\T)$ and $j_0 \in J$. Then we can extend $p$ to $q$ with $\dom(q) = F_p \times \finseqleqomegaarg{n_q}$, so that $(q, w_p) \in \TT(\T)$, $(q, w_p) \leq (p, w_p)$ and we have that $[T^i_q] \cap [T^j_q] = \emptyset$ as well as $[T^i_q] \cap [T_{j_0} \restr n_q] = \emptyset$ for all $i \neq j \in F_p$.
\end{proposition}
\begin{proof}
Choose $m > n_p$ so that $f \restr m \neq g \restr m$ for all $f \neq g \in H_p$ and $f \restr m \notin T_{j_0}$ for all $f \in H_p$. By  Proposition~\ref{prop2.9} extend $p$ to $q_0$ with $\dom(q_0) = F_p \times \finseqleqomegaarg{m - 1}$.
	
For every $i \in F_p$ and $s \in [T^i_{q_0}]$ the set
$$K^i_s = \set{k < \omega}{s \concat k \in T_{j_0} \text{ or } s \concat k = f \restr m \text{ for some } f\in H_p \text{ with } w_p(f) \neq i}$$
 is finite, so there are pairwise different $\set{k^i_s < \omega}{i \in F_p, s \in [T^i_{q_0}]}$ such that $k^i_s \notin K^i_s$.
	
We define $q$ for $i \in F_p$ and $t \in \finseqleqomegaarg{m}$ by
	$$
	q(i)(t) =
	\begin{cases}
	q_0(i)(t) & \text{if } t \in \finseqleqomegaarg{m-1} \\
	1 & \text{if } t = s \concat k \text{ with } s \in [T^{i}_{q_0}] \text{ and } k = k^i_{s}\\
	1 & \text{if } t = f \restr m \text{ with } f \in H_p \text{ with } w_p(f) = i\\
	0 & \text{otherwise}
	\end{cases}
	$$
Then $T^{i}_q$ is a finitely splitting $m$-tree for every $i \in F_p$, $\dom(q) = F_p \times \finseqleqomegaarg{m}$, $(q, w_p) \in \TT(\T)$ and $(q, w_p) \extends (p, w_p)$. It remains to show that $(q, w_p)$ has the desired properties, so let $i \neq j \in F_p$.
	
Let $s \concat k \in [T^i_q] \cap [T^j_q]$. Assume $k = k^j_s$.
Then $k^i_s \neq k^j_s = k$, so choose $f \in H_p$ with $w_p(f) = i$ and $f \restr m = s \concat k$. Then $k \in K^j_s$, which contradicts $k = k^j_s \notin K^j_s$. The case $k = k^i_s$ follows analogously, so assume $k^i_s \neq k \neq k^j_s$. Choose $f,g \in H_p$ with $w_p(f) = i, w_p(g) = j$ and $f \restr m = s \concat k = g \restr m$. This contradicts the choice of $m$.
	
Finally, let $s \concat k \in [T^i_q] \cap [T_{j_0} \restr m]$. Then $k \in K^i_s$, so $k \neq k^i_s$. Now, choose $f \in H_p$ with $w_p(f) = i$ and $f \restr m = s \concat k$. Again, this contradicts the choice of $m$.
\end{proof}

\begin{lemma}\label{generic_adfs}
	Let $\T = \set{T_j}{j \in J}$ be an a.d.f.s.\ family and let $G$ be $\TT(\T)$-generic.
	In $V[G]$ we have that $\T \cup \S_{G}$ is an a.d.f.s.\ family.
\end{lemma}

\begin{proof}
Let $i \in \omega$. First, we show that $S_{G, i}$ is a finitely splitting tree in $V[G]$. Let $t \in S_{G,i}$ and $s \subseteq t$.
Then we can choose $(p, w_p) \in G$ with $p(i)(t) = 1$. But $T^i_p$ is an $n_p$-tree, so also $p(i)(s) = 1$, i.e.\ $s \in S_{G_i}$ which shows that $S_{G_i}$ is a tree.
	
Next, we show that $S_{G,i}$ is finitely splitting, so let $s \in S_{G_i}$. Again, choose $(p, w_p) \in G$ with $p(i)(s) = 1$.
By Proposition~\ref{prop2.9} there is $(q, w_p) \in G$ with $(q, w_p) \extends (p, w_p)$ and $\dom(q) = F_p \times \finseqleqomegaarg{n_q}$ for $n_q > |s|$. But then $T^{i}_q$ is an $n_q$-tree and $s \in T^{i}_q \setminus [T^i_q]$, which implies that $\suc_{S_{G,i}}(s) = \suc_{T^i_q}(s)$, so $S_{G,i}$ is finitely splitting in $V[G]$.
	
Let $i_0 \neq j_0 \in \omega$. We show that $S_{G,i_0}$ and $S_{G,j_0}$ are almost disjoint in $V[G]$. Let $(p, w_p) \in \TT(\T)$.
By Proposition~\ref{prop2.8} there is $q$ extending $p$ such that $\dom(q) = (F_p \cup \simpleset{i_0, j_0}) \times \finseqleqomegaarg{n_p}$, so that $(q, w_p) \in \TT(\T)$ and $(q, w_p) \extends (p, w_p)$. By Proposition~\ref{prop2.10} there is an extension  $r$ of $q$ with $\dom(r) = (F_p \cup \simpleset{i_0,j_0}) \times \finseqleqomegaarg{n_r}$, so that $(r, w_p) \in \TT(\T)$, $(r, w_p) \extends (q, w_p)$ and $[T^i_r] \cap [T^j_r] = \emptyset$ for all $i \neq j \in F_p \cup \simpleset{i_0, j_0}$. But then we have that
$$(r, w_p) \forces [S_{\gen,i_0} \restr n_r] \cap [S_{\gen,j_0} \restr n_r] = [T^{i_0}_r] \cap [T^{j_0}_r] = \emptyset$$ so the set of all conditions which force that $S_{\gen, i_0}$ and $S_{\gen, j_0}$ are almost disjoint is dense.
	
Finally, let $i_0 \in \omega$ and $j \in J$. We show that $S_{G,i_0}$ and $T_j$ are almost disjoint in $V[G]$. Consider $(p, w_p) \in \TT(\T)$.
By Proposition~\ref{prop2.8} there is an extension $q$ of $p$ with $\dom(q) = (F_p \cup \simpleset{i_0}) \times \finseqleqomegaarg{n_p}$, so that $(q, w_p) \in \TT(\T)$ and $(q, w_p) \extends (p, w_p)$.
By Proposition~\ref{prop2.10} we can extend $q$ to $r$ with $\dom(r) = (F_p \cup \simpleset{i_0}) \times \finseqleqomegaarg{n_r}$, so that $(r, w_p) \in \TT(\T)$ and $(r, w_p) \extends (q, w_p)$ and $[T^i_r] \cap [T_j \restr n_r] = \emptyset$ for all $i \in F_p \cup \simpleset{i_0}$.
But then 
$$(r, w_p) \forces [S_{\gen,i_0} \restr n_r] \cap [T_j \restr n_r] = [T^{i_0}_r] \cap [T_j \restr n_r] = \emptyset$$
so the set of all conditions which force that $S_{\gen, i_0}$ and $T_j$ are almost disjoint is dense.
\end{proof}

\begin{proposition}\label{diag}
	Let $\T$ be an a.d.f.s.\ family and let $G$ be $\TT(\T)$-generic.
	Then in $V[G]$ we have that  $\S_G$ diagonalizes $\T$.
\end{proposition}

\begin{proof}
	Let $(p, w_p) \in \TT(\T)$ and $f \in \W(\T)$.
	If $f \notin H_p$ we can choose an $i \in \omega \setminus F_p$ and consider $(q,w_q)=(p, w_p \cup (f, i)) \extends (p, w_p)$. 
	Then, however, $f\in H_q$ and $(q, w_q) \forces ``f \in [S_{\gen, w_q(f)}]"$.
\end{proof}

Using this diagonalization property we obtain a maximal a.d.f.s.\ family as follows:

\begin{theorem} \label{key}
Let $\kappa$ be a cardinal of uncountable cofinality.
Let $\seq{\TT_\alpha, \dot{\QQ}_\gamma}{\alpha \leq \kappa,\gamma<\kappa}$ be the finite support iteration, where $\dot{\QQ}_\alpha = \TT(\dot{\T}_\alpha)$ and $\dot{\T}_\alpha$ is a $\TT_\alpha$-name for $\bigcup_{\beta < \alpha}\S^\beta_{\gen_{\beta + 1} / \gen_{\beta}}$.
Let $G$ be $\TT_\kappa$-generic. Then $\T_\kappa = \bigcup_{\alpha < \kappa} \S^\alpha_G$ is a maximal a.d.f.s.\ family.
\end{theorem}

\begin{proof}\label{iterateformaxadfs}
By Lemma~\ref{generic_adfs}, $\T_\kappa$ is an a.d.f.s.\ family in $V[G]$.
Assume there was a finitely splitting tree $T$ almost disjoint from $\T_\kappa$. As $\TT_\kappa$ is c.c.c.\ we can choose $\alpha < \kappa$ such that $T \in V[G_\alpha]$, where $G_\alpha = G \cap \TT_\alpha$ is a $\TT_\alpha$-generic filter. By assumption $T \in \I^+(\T_\alpha)$, so by Proposition~\ref{diag} in $V[G_{\alpha + 1}]$ there is an $i < \omega$ such that $T$ has non-empty intersection with $S^{\alpha + 1}_{G_{\alpha + 1} / G_\alpha, i}$, which is a contradiction.
\end{proof}

In particular we obtain the following consistency result:
\begin{theorem}\label{thm.aT.c.arbitrary}
	Assume CH and let $\kappa \leq \lambda$ be regular cardinals.
	Then there is c.c.c.\ extension in which  $\d = \aT = \kappa \leq \lambda = \c$ holds.
\end{theorem}

\begin{proof}
	Use $\CC_\lambda \forcingconcat \dot{\TT}_\kappa$, where $\CC_\lambda$ is $\lambda$-Cohen forcing and $\dot{\TT}_\kappa$ is the forcing from the previous theorem.
	In the generic extension clearly $\lambda = \c$ holds and we have $\aT \leq \kappa$ by the previous theorem.
	Further, $\d \geq \kappa$, because $\dot{\TT}_\kappa$ adds Cohen reals cofinally often, which follows either from the fact the we use finite support and add Cohen reals at limit steps of countable cofinality or from fact that already $\TT(\T)$ adds Cohen reals, which we will prove at the end of this section.
\end{proof}

We can also add many maximal a.d.f.s.\ families at the same time with our forcing.
We use an analogous construction as Fischer and Shelah in \cite{fischershelah}, where they add many maximal independent families at the same time.

\begin{theorem}\label{manyadfs}
	Assume GCH.
	Let $\lambda$ be a cardinal of uncountable cofinality, $\theta \leq \lambda$ a regular cardinal and $\seq{\kappa_\beta}{\beta < \theta}$ a sequence of regular uncountable cardinals with $\cof(\lambda) \leq \kappa_\beta \leq \lambda$ for all $\beta < \theta$.
	Then there is a c.c.c.\ extension in which $\c = \lambda$ and $\kappa_\beta \in \spec(\aT)$ for all $\beta < \theta$.
\end{theorem}

\begin{proof}
As $\cof(\lambda) \leq \kappa_\beta$ for all $\beta < \theta$ we may choose a partition of (possibly a subset of) $\lambda$ into $\theta$-many disjoint sets $\seq{I_\beta}{\beta < \theta}$ such that $|I_\beta| = \kappa_\beta$ and $I_\beta$ is cofinal in $\lambda$ for all $\beta < \theta$. We define a finite support iteration $\seq{\PP_\alpha, \dot{\QQ}_\alpha}{\alpha < \lambda}$ of c.c.c.\ forcings as follows: We will iteratively add $\theta$-many a.d.f.s.\ families $\seq{\T_\beta}{\beta < \theta}$. Denote with $\dot{\T}^\alpha_\beta$ the name for the $\beta$-th family after iteration step $\alpha$.
Initially we let $\dot{\T}^0_\beta$ be a name for the empty set for all $\beta < \theta$. Now, assume that $\PP_\alpha$ and $\dot{\T}^\alpha_\beta$ have been defined for all $\beta < \theta$. If there is no $\beta <  \theta$ such that $\alpha \in I_\beta$ let $\dot{\QQ}_\alpha$ be a name for the trivial forcing. Otherwise, let $\beta_\alpha < \theta$ be the unique index such that $\alpha \in I_{\beta_\alpha}$ and let $\dot{\QQ}_\alpha$ be a name for $\TT(\dot{\T}^\alpha_{\beta_\alpha})$. Furthermore, for $\beta \neq \beta_\alpha$ let $\dot{\T}^{\alpha + 1}_\beta$ be a name for the same family as $\dot{\T}^{\alpha}_\beta$ and for $\beta_\alpha$ let $\dot{\T}^{\alpha + 1}_\beta$ be a name for $\dot{\T}^{\alpha}_\beta \cup \S_\gen$.
	
Let $G$ be $\PP_\lambda$-generic. As $I_\beta$ is cofinal in $\lambda$ for all $\beta < \theta$ and by the results of the previous section we obtain in $V[G]$ that $\T_\beta$ is a maximal a.d.f.s.\ family for all $\beta < \theta$. Furthermore, as $|I_\beta| = \kappa_\beta$ and every iteration step extends at most one family by $\omega$-many new trees we obtain that $\T_\beta$ has size $\kappa_\beta$ for all $\beta < \theta$. But this shows that $\kappa_\beta \in \spec(\aT)$. Finally, GCH, the c.c.c.-ness of all $\PP_\alpha$ and $\cof(\lambda) > \aleph_0$ imply that $\c = \lambda$ by counting nice names.
\end{proof}

\begin{proposition} \label{cohen}
	$\TT(\T)$ adds Cohen reals.
\end{proposition}

\begin{proof}
	Let $G$ be $\TT(\T)$-generic.
	Define a real $c:\omega \to 2$ by $c(i) = 1$ iff $\simpleseq{0} \in S_{G, i}$.
	We show that $c$ defines a Cohen real over $V$.
	In $V$ let $D \subseteq \CC$ be dense and $(p, w_p) \in \TT(\T)$.
	By Propositions \ref{prop2.8} and \ref{prop2.9} we may assume that $n_p > 0$ and $F_p = [0, N]$ and $\ran(H_p) \subseteq F_p$ for some $N < \omega$.
	
	Define $s(i) = 1$ iff $\simpleseq{0} \in T_p^i$ for $i \leq N$.
	By definition of $c$ we get	$(p, w_p) \forces ``s \subseteq \dot{c}"$.
	By density of $D$ choose $t \in D$ such that $s \subseteq t$.
	Let $S,T$ be any $n_p$-trees, such that $\simpleseq{0} \in S$ and $\simpleseq{0} \notin T$.
	Then we extend $(p, w_p)$ to $(q,w_p)$ with $\dom(q) = |t| \times \finseqleqomegaarg{n_p}$ by
	$$
	T^i_q =
	\begin{cases}
	T^i_p & \text{if } i \leq N \\
	S & \text{if } i > N \text{ and } t(i) = 1 \\
	T & \text{otherwise}
	\end{cases}
	$$
	But then we get that $(q, w_p) \forces ``t \subseteq \dot{c}"$, where $t \in D$.
\end{proof}

\begin{remark} \label{unbounded}
Let $\T$ be an a.d.f.s.\ family. Apart from the Cohen reals above, there is also a second kind of unbounded reals added by $\TT(\T)$. Let $G$ be a $V[G]$-generic filter. For $i_0 < \omega$ in $V[G]$ define
$$f_{i_0}(n) = \max\set{k < \omega}{\text{there is an } s \in S_{G,i_0} \text{ with } s(n) = k}$$ 
To show that $f_{i_0}$ is unbounded over $V$, consider any $g \in \infseqomega \cap V$, $n_0 < \omega$ and $(p, w_p) \in \TT(\T)$.
By Propositions \ref{prop2.8} and \ref{prop2.9} we may assume that $i_0 \in F_p$ and $n_0 \leq n_p$. As before, given $i \in F_p$ we can choose an ($n_p + 1$)-tree $T_i$ which extends $T^i_p$ and contains $f \restr (n_p + 1)$ for all $f \in H_p$ with $w_p(f) = i$, but for $i_0$ additionally also contains an $s \in T_{i_0}$ such that $s(n_p) = g(n_p) + 1$.
Define $q:F_p \times \finseqleqomegaarg{n_p + 1}$ by $T^i_q = T_i$. Then we have $(q, w_p) \extends (p, w_p)$ and $(q, w_p) \forces g(n_p) < s(n_p) \leq f_{i_0}(n)$.
\end{remark}

Finally, we show that $\TT(\T)$ can both add and not add dominating reals, depending on the properties of $\T$.
Note that the following characterization is sufficient but not necessary.

\begin{proposition} \label{dominating}
Let $\T$ be an a.d.f.s.\ family.  If $\I^+(\T)$ is a dominating family then $\TT(\T)$ adds a dominating real. In particular $\TT(\emptyset)$ adds dominating reals.
\end{proposition}

\begin{proof}
	Assume $\I^+(\T)$ is dominating.
	Let $G$ be $\TT(\T)$-generic.
	In $V[G]$ choose $f$ such that $f_i \leq^* f$, where $f_i$ is defined as in the previous remark.
	We claim that $f$ is dominating, so in $V$ let $g \in \infseqomega$ and $(p, w_p) \in \TT(\T)$.
	Choose $h \in \I^+(\T)$ and $N < \omega$ such that for all $n \geq N$ we have $g(n) \leq h(n)$.
	By possibly extending $(p, w_p)$ we may assume that $h \in H_p$.
	Let $i = w_p(h)$ and choose $(q, w_q) \extends (p, w_p)$ and $M < \omega$ with $N \leq M < \omega$ such that
	$$
	(q, w_q) \forces \text{For all } n \geq M \text{ we have } f_i(n) \leq f(n)
	$$
	But then $(q, w_q) \forces g \leq^* f$ as $w_p(h) = i$ implies that
\[		(q, w_q) \forces \text{For all } n \geq M \text{ we have } g(n) \leq h(n) \leq f_i(n) \leq f(n)
\]
\end{proof}

\begin{remark}
	On the other hand, if $\T$ is a maximal a.d.f.s.\ family, we have $\TT(\T) \cong \CC$, so that $\TT(\T)$ does not add dominating reals.
\end{remark}

\begin{question}
	Is there a nice combinatorial characterization of those families $\T$ for which $\TT(\T)$ adds a dominating real?
\end{question}

\section{A Sacks-indestructible maximal a.d.f.s.\ family}

In the last section we have seen that $\spec(\aT)$ can be arbitrarily large, in fact it may be a superset of any set of regular uncountable cardinals.
Contrarily, in this section we will prove that CH implies the existence of a maximal a.d.f.s.\ family which stays maximal even after forcing with countably supported iteration or product of Sacks forcing of any length.
Together with an ``isomorphism of names"-argument we will then that the Sacks-product extension satisfies $\spec(\aT) = \simpleset{\aleph_1, \c}$ where $\c$ may be any cardinal of uncountable cofinality.
Thus, consistently $\aT = \aleph_1$ and $\spec(\aT)$ is minimal.

We begin with a brief reminder of the definitions of Sacks forcing and countably supported Sacks forcing and their fusion sequences.

\begin{definition}
	Let $T \subseteq 2^{<\omega}$ be a tree.
	\begin{enumerate}
		\item Let $s \in T$ then $T_s = \set{t \in T}{s \subseteq t \text{ or } t \subseteq s}$.
		\item $\spl(T) = \set{s \in T}{s \concat 0 \in T \text{ and } s \concat 1 \in T}$ is the set of all splitting nodes of $T$.
		\item $T$ is perfect iff for all $s \in T$ there is $t \in \spl(T)$ such that $s \subseteq t$.
		\item $\SS = \set{T \subseteq 2^{<\omega}}{T\text{ is a perfect tree}}$ ordered by inclusion is Sacks forcing.
	\end{enumerate}
\end{definition}

\begin{definition}
	Let $T \in \SS$.
	We define the fusion ordering for Sacks forcing as follows:
	\begin{enumerate}
		\item Let $s \in T$ then $\sucspl_T(s)$ is the unique minimal splitting node in $T$ extending $s$.
		\item $\stem(T) = \sucspl_T(\emptyset)$.
		\item $\spl_0 = \simpleset{\stem(T)}$ and if $\spl_n(T)$ is defined for $n < \omega$ we set
		$$
		\spl_{n + 1} = \set{\sucspl_T(s \concat i)}{s \in \spl_n(T), i \in 2}
		$$
		$\spl_n(T)$ is called the $n$-th splitting level of $T$.
		\item Let $n < \omega$ and $S, T \in \SS$.
		We write $S \extends_n T$ iff $S \extends T$ and $\spl_n(S) = \spl_n(T)$.
	\end{enumerate}
\end{definition}

The following lemmata are well-known, see \cite{kanamori} for example. We add them for completeness.

\begin{lemma}
	Let $\seq{T_n \in \SS}{n < \omega}$ be a sequence of trees such that $T_{n + 1} \extends_n T_n$ for all $n < \omega$.
	Then $T = \bigcap_{n < \omega}T_n \in \SS$ and $T \extends_n T_n$ for all $n < \omega$.
\end{lemma}

\begin{proof}
	The only non-trivial property to verify is that $T \in \SS$, so let $s \in T$.
	So $s \in T_0$ and let $t = \sucspl_{T_0}(s)$ and $n < \omega$ such that $t \in \spl_n(T_0)$.
	As $T_n \extends T_0$ choose $u \in \spl_n(T_n)$ such that $t \subseteq u$.
	But then $s \subseteq u$ and by definition of $\leq_n$ we get that $u \in \spl(T)$.
\end{proof}

\begin{definition}
	Let $\lambda$ be a cardinal.
	$\SS^\lambda$ is the countable support product of Sacks forcing of size $\lambda$. Moreover,
	\begin{enumerate}
		\item for $A \subseteq \SS^\lambda$ let $\bigcap A$ be the function with $\dom(\bigcap A) = \bigcup_{p \in A} \dom(p)$ and for all $\alpha < \lambda$ we have $(\bigcap A)(\alpha) = \bigcap_{p \in A} p(\alpha)$.
		Notice that we do not necessarily have $\bigcap A \in \SS^\lambda$.
		\item Let $n < \omega$, $p, q \in \SS^\lambda$ and $F \in \finsubset{\dom(q)}$.
		Write $p \extends_{F, n} q$ iff $p \extends q$ and $p(\alpha) \extends_n q(\alpha)$ for all $\alpha \in F$.
	\end{enumerate}
\end{definition}

\begin{lemma}
	Let $\seq{p_n \in \SS^\lambda}{n < \omega}$ and $\seq{F_n}{n < \omega}$ be sequences such that
	\begin{enumerate}
		\item $p_{n + 1} \extends_{F_n, n} p_n$ for all $n < \omega$
		\item $F_{n} \subseteq F_{n + 1}$ for all $n < \omega$ and $\bigcup_{n < \omega} F_n = \bigcup_{n < \omega} \dom(p_n)$
	\end{enumerate}
	Then $p = \bigcap_{n < \omega} p_n \in \SS^\lambda$ and $p \extends_{F_n, n} p_n$ for all $n < \omega$.
\end{lemma}

\begin{proof}
	Again, we only need to verify that $p \in \SS^\lambda$.
	Clearly, $\dom(p) = \bigcup_{n < \omega} \dom(p_n)$ is countable.
	Let $\alpha \in \dom(p)$.
	Choose $N < \omega$ such that $\alpha \in F_N$.
	By assumption we get that $\seq{p_n(\alpha) \in \SS}{n \geq N}$ is a fusion sequence, so $p(\alpha) = \bigcap_{n < \omega} p_n(\alpha) \in \SS$ by the previous lemma.
\end{proof}

\begin{definition}
	Let $p \in \SS^\lambda$, $F \in \finsubset{\dom(p)}$, $n < \omega$ and $\sigma:F \to V$ be a suitable function for $p$, $F$ and $n$, i.e.\ for all $\alpha \in F$ there are $s \in \spl_n(p(\alpha))$ and $i \in 2$ with $\sigma(\alpha) = s \concat i$.
	Then we define $p \restr \sigma \in \SS^\lambda$ by
	$$
	(p \restr \sigma) (\alpha) = 
	\begin{cases}
	p(\alpha)_{\sigma(\alpha)} & \text{if } \alpha \in F \\
	p(\alpha) & \text{otherwise}
	\end{cases}
	$$
	Notice that for fixed $p \in \SS^\lambda$, $n < \omega$ and $F \in \finsubset{\dom(p)}$ there are only finitely many $\sigma$ which are suitable for $p$, $F$ and $n$.
	Also, if $q \extends_{F, n} p$, then $q$ and $p$ have the same suitable functions for $F$ and $n$.
	Furthermore, the set
	$$
	\set{p \restr \sigma}{\sigma:F \to V \text{ is a suitable function for } p, F \text{ and } n}
	$$
	is a maximal antichain below $p$.
\end{definition}

\begin{lemma}\label{F_n_density}
	Let $p \in \SS^\lambda$, $D \subseteq \SS^\lambda$ be dense open below $p$, $n < \omega$ and $F \in \finsubset{\dom(p)}$.
	Then there is $q \extends_{F, n} p$ such that for all $\sigma$ suitable for $p$ (or equivalently $q$), $F$ and $n$ we have $q \restr \sigma \in D$.
\end{lemma}
\begin{proof}
Let $\seq{\sigma_i}{i < N}$ enumerate all suitable functions for $p$, $F$ and $n$. Set $q_0 = p$. We will define a $\extends_{F, n}$-decreasing sequence $\seq{q_i}{i \leq N}$ so that all of the $q_i$ have the same suitable functions as $p$ for $F$ and $n$.
Assume $i < N$ and $q_i$ is defined. Choose $r_i \extends q_i \restr \sigma_i$ in $D$ and define
	$$
	q_{i + 1}(\alpha) =
	\begin{cases}
	r_i(\alpha) \cup \bigcup \set{q_i(\alpha)_s}{s = t \concat i \text{ for } t \in \spl_n(p), i \in 2 \text{ and } s \neq \sigma(\alpha)} & \text{if } \alpha \in F \\
	r_i(\alpha) & \text{otherwise}
	\end{cases}
	$$
Clearly, $q_{i + 1} \extends_{F, n} q_i$ and $q_{i + 1} \restr \sigma = r_i$. Now, set $q = q_N$ and let $\sigma$ be suitable for $p$, $F$ and $n$. Choose $i < N$ such that $\sigma = \sigma_i$. Then we have $q \restr \sigma \extends q_{i+1} \restr \sigma = r_i \in D$, so $q \restr \sigma \in D$ as $D$ is open.
\end{proof}

By routine fusion arguments one obtains that both $\SS$ and $\SS^\lambda$ are proper and $\infseqomega$-bounding.
Here, we will need the following key lemma:

\begin{lemma} \label{extendctbladfs}
Let $\T$ be a countable a.d.f.s.\ family, $\lambda$ be a cardinal, $p \in \SS^\lambda$ and $\dot{f}$ be a $\SS^\lambda$-name for a real such that for all $T \in \T$ we have
$$p \forces \dot{f} \notin [T].$$
Then there is a finitely splitting tree $S$ and $q \extends p$ such that $\T \cup \simpleset{S}$ is an a.d.f.s.\ family and
$$q \forces \dot{f} \in [S].$$
\end{lemma}
\begin{proof}
Enumerate $\T = \set{T_n}{n < \omega}$. By assumption for every $n < \omega$ the set
$$D_n = \set{r \in \SS^\lambda}{\text{There is a } k <\omega \text{ such that } r \forces \dot{f} \restr k \notin T_n}$$
is open dense below $p$. Set $q_0 = p$. We define a fusion sequence $\seq{q_n}{n < \omega}$ as follows. Fix with a suitable bookkeeping argument a sequence $\seq{F_n \in \finsubset{\dom(q_n)}}{n < \omega}$ such that that $\bigcup_{n < \omega} F_n = \bigcup_{n < \omega} \dom(q_n)$. Now, let $n < \omega$ and assume $q_n$ has been defined, and $F_n$ is given. Apply Lemma~\ref{F_n_density} to $D_n$, $q_n$, $n$ and $F_n$ to obtain $q_{n + 1} \extends_{F_n,n} q_n$ such that for all suitable functions $\sigma$ for $q_n$, $F_n$ and $n$ there is a $k_\sigma < \omega$ such that
$$q_{n + 1} \restr \sigma \forces \dot{f} \restr k_\sigma \notin T_n.
$$
Set $k_n = \max\set{k_\sigma}{\sigma \text{ is a suitable function for } q_n, F_n \text{ and } n}$. Then we have
$$
q_{n + 1} \forces \dot{f} \restr k_n \notin T_n.
$$
Let $q_\omega = \bigcap_{n < \omega} q_n$. Finally, $\SS^\lambda$ is $\infseqomega$-bounding, so we may choose $q \extends q_\omega$ such that
$$S = \set{s \in \finseqomega}{\text{there is an } r \extends q \text{ with } r \forces s \subseteq \dot{f}}$$
is a finitely splitting tree.
	Clearly, $q \forces \dot{f} \in [S]$ by definition of $S$.
	Furthermore, for $n < \omega$ we have that $S$ and $T_n$ are almost disjoint, since $q \extends q_{n + 1}$, $q_{n + 1} \forces \dot{f} \restr k_n \notin T_n$ and the definition of $S$ imply
$$q \forces [T \restr k_n] \cap [S \restr k_n] = \emptyset.$$
\end{proof}

Finally, we will need that $\SS^\lambda$ satisfies a nice version of continuous reading of names as proved in \cite{sacksmed}.
To state this version of continuous reading we summarize the most important definitions with minor tweaks for simplicity.
First, we modify the presentation of~\cite[Definition 2.5 and Proposition 2.6]{kechris} to code continuous functions $f^*:{^\omega (^\omega 2)} \to \infseqomega$:

\begin{definition}\
	\begin{enumerate}
		\item For $s, t \in {^{<\omega}(^{<\omega}2)}$ write $s \trianglelefteq t$ iff $\dom(s) \subseteq \dom(t)$ and for all $n \in \dom(s)$, $s(n) \subseteq t(n)$.
		\item A function $f:{^{<\omega}(^{<\omega}2)} \to \finseqomega$ is monotone if for all $s \trianglelefteq t \in {^{<\omega}(^{<\omega}2)}$, $f(s) \subseteq f(t)$.
		\item A function $f:{^{<\omega}(^{<\omega}2)} \to \finseqomega$ is proper iff for all $x \in {^\omega (^\omega 2)}$:
		$$
		|\dom(f(x \restr n \times n))| \overset{n \to \infty}{\longrightarrow} \infty.
		$$
		\item For a monotone, proper function $f:{^{<\omega}(^{<\omega}2)} \to \finseqomega$ define a continuous function
		$$
		 f^*:{^\omega (^\omega 2)} \to \infseqomega\hbox{ via }
		f^*(x) = \bigcup_{n < \omega} f(x \restr n \times n).
		$$
		In this case $f$ is called a code for $f^*$.
	\end{enumerate}
\end{definition}

\begin{remark}
	Conversely, for every continuous function $f^*:{^\omega (^\omega 2)} \to \infseqomega$ there is a code for it.
	In fact, this is true in general for continuous functions between any two effective Polish spaces.
\end{remark}

\begin{remark}
	For all $p, q \in \SS$ there is a natural bijection $\pi:\spl(p) \to \spl(q)$ which for every $n < \omega$ restricts to bijections $\pi \restr \spl_n(p): \spl_n(p) \to \spl_n(q)$ and which preserves the lexicographical ordering.
	We can extend it to a monotone and proper function $\pi:p \to q$ in a similar sense as above.
	$\pi$ then codes a homeomorphism $\pi:[p] \to [q]$ which we call the induced homeomorphism.
\end{remark}

\begin{definition}
	Let $\PP$ be the countable support iteration of Sacks forcing of length $\lambda$.
	Let $p \in \PP$ and assume that $|\dom(p)| = \omega$ and $0 \in \dom(p)$.
	Note that this assumption is clearly dense in $\PP$.
	\begin{enumerate}
		\item A standard enumeration of $\dom(p)$ is a sequence
		$$
		\Sigma = \seq{\sigma_k}{k < \omega}
		$$
		such that $\sigma_0 = 0$ and $\ran(\Sigma) = \dom(p)$.
		\item Let $[p]$ be a $\PP$-name such that
		$$
		\PP \forces [p] = \seq{x \in {^{\dom(p)}}({^\omega} 2)}{\text{For all } \alpha \in \dom(p) \text{ we have } x(\alpha) \in [p(\alpha)]}
		$$
		\item Let $\Sigma$ be a standard enumeration of $\dom(p)$.
		For $k < \omega$ let $\dot{e}_k^{p, \Sigma}$ be a $\PP \restr \sigma_k$-name such that
		$$
		\PP \restr \sigma_k \forces \dot{e}_k^{p, \Sigma} \text{ is the induced homeomorphism between } [p(\sigma_k)] \text{ and } ^\omega 2
		$$
		Finally, let $\dot{e}^{p, \Sigma}$ be a $\PP$-name such that
		$$
		\PP \forces \dot{e}^{p, \Sigma}:[p] \to {^\omega(^\omega 2)} \text{ such that } \dot{e}^{p, \Sigma}(x) = \seq{\dot{e}^{p, \Sigma}_k(x(\sigma_k))}{k < \omega} \text{ for all } x \in [p]
		$$
	\end{enumerate}
\end{definition}

\begin{remark}
For the countable support product of Sacks forcing we define the analogous notions. In fact in this simpler case we do not have to define $[p]$, $\dot{e}^{p, \Sigma}_k$ and $\dot{e}^{p, \Sigma}$ as names.
\end{remark}

\begin{definition}
Let $\PP$ be the countable support iteration or product of Sacks forcing of any length.
	Let $q \in \PP$ and $\dot{f}$ be a $\PP$-name such that $q \forces ``\dot{f} \in \infseqomega"$.
	Let $\Sigma = \seq{\sigma_k}{k < \omega}$ be a standard enumeration of $\dom(q)$ and $f:{^{<\omega}(^{<\omega}2)} \to \finseqomega$ be a code for a continuous function $f^*:{^\omega (^\omega 2)} \to \infseqomega$ such that
	$$
	q \forces \dot{f} = (f^* \circ \dot{e}^{q, \Sigma})(s_\gen \restr \dom(q))
	$$
	where $s_\gen$ is the sequence of Sacks reals.
	Then we say $\dot{f}$ is read continuously below $q$ (by $f$ and $\Sigma$).
\end{definition}

\begin{lemma}[Lemma 4 of \cite{sacksmed}]
	Let $\PP$ be the countable support iteration or product of Sacks forcing of length $\lambda$.
	Suppose $p \in \PP$ and $\dot{f}$ is a $\PP$-name such that $p \forces ``\dot{f} \in \infseqomega"$.
	Then there is $q \extends p$ such that $\dot{f}$ is read continuously below $q$.
\end{lemma}

\begin{remark}
	For any $p \in \PP$ and $\PP$-name $\dot{f}$ such that $p \forces ``\dot{f} \in \infseqomega"$ it is easy to see that if $\dot{f}$ is read continuously below $p$ then for all $q \extends p$ also $\dot{f}$ is read continuously below $q$.
	Thus, the previous lemma shows that the set
	$$
	\set{q \in \PP}{\dot{f} \text{ is read continuously below } q}
	$$
	is dense open below $p$.
\end{remark}

\begin{theorem}\label{indestructible_adfs}[CH]
There is an a.d.f.s.\ family which remains maximal after forcing with countably supported iteration or product of Sacks forcing of arbitrary length.
\end{theorem}

\begin{proof}
By CH let $\seq{f_\alpha}{\alpha < \aleph_1}$ enumerate all codes $f_\alpha:{^{<\omega}(^{<\omega}2)} \to \finseqomega$ for continuous functions $f_\alpha^*:{^\omega (^\omega 2)} \to \infseqomega$. We define an increasing sequence $\seq{T_\alpha}{\alpha < \aleph_1}$ of a.d.f.s.\ families as follows:
	
Set $\T_0 = \emptyset$. Now, assume $\T_\alpha$ is defined and is countable. If for all $T \in \T_\alpha$ we have that
$$\SS^{\aleph_0} \forces f^*(s_\gen) \notin [T]$$
then by Lemma \ref{extendctbladfs} there is a finitely splitting tree $S$ and $p \in \SS^{\aleph_0}$ such that $\T_\alpha \cup \simpleset{S}$ is an a.d.f.s.\ family and
$$p \forces f^*(s_\gen) \in [S].$$
Set $\T_{\alpha + 1} = \T_\alpha \cup \simpleset{S}$. In the other case we set $\T_{\alpha + 1} = \T_\alpha$.
	
This finishes the construction and we set $\T = \bigcup_{\alpha < \aleph_1}\T_{\alpha}$. By construction we have that for any code $f:{^{<\omega}(^{<\omega}2)} \to \finseqomega$ for a continuous function $f^*:{^\omega (^\omega 2)} \to \infseqomega$ there is a $p \in \SS^{\aleph_0}$ and $T \in \T$ such that
$$p \forces f^*(s_\gen) \in [T].$$
We claim that this implies that for all $x \in [p]$, $f^*(x) \in [T]$, for if $x \in [p]$ and $n < \omega$ we define $p_{x \restr n \times n} \extends p$ as follows:
For $m < \omega$ let
$$p_{x \restr n \times n}(m) =
	\begin{cases}
	(p(m))_{x_m \restr n} & \text{if } m < n \\
	p(m) & \text{otherwise.}
	\end{cases}$$
This is well-defined and $p_{x \restr n \times n} \in \SS^{\aleph_0}$ since $x_m \restr n \in p(m)$. But then we have
$$p_{x \restr n \times n} \forces f(s_\gen \restr n \times n) \in T \text{ and } s_\gen \restr n \times n = x \restr n \times n $$
which yields that $f(x \restr n \times n) \in T$. Thus, we have shown $f^*(x) \in [T]$.
	
Assume that $\T$ is not maximal in some iterated Sacks-extension with countable support. The argument for countably supported product of Sacks forcing follows similarly. So let $\lambda$ be an ordinal and $\PP$ the iterated Sacks forcing of length $\lambda$; we may assume that $\lambda \geq \omega$. Further, let $\bar{q} \in \PP$ and $\dot{f}$ be a $\PP$-name for a real such that for all $T \in \T$ we have
$$\bar{q} \forces_{\PP} \dot{f} \notin [T].$$
Choose $q \extends \bar{q}$ and a standard enumeration $\Sigma = \seq{\sigma_k}{k \in \omega}$ of $\dom(q)$ and $f:{^{<\omega}(^{<\omega}2)} \to \finseqomega$ a code for a continuous function $f^*:{^\omega (^\omega 2)} \to \infseqomega$ such that
$$q \forces_\PP \dot{f} = f^*(\dot{e}^{q, \Sigma}(s_\gen)).$$
By construction of $\T$ we may choose $T \in \T$ and $p \in \SS^{\aleph_0}$ such that for all $x \in [p]$ we have $f^*(x) \in [T]$. Let $r$ be the \textquoteleft pull-back\textquoteright of $p$ under $\dot{e^{q, \Sigma}}$, i.e.\ for all $k < \omega$ we have
$$\forces_{\PP} [r(\sigma_k)] = (\dot{e}^{q, \Sigma}_k)^{-1}[[p(k)]] \quad \text{ so that } \quad  r \forces_\PP \dot{e}^{q, \Sigma}(s_\gen) \in [p].$$
By definition of $\dot{e}^{q, \Sigma}$ we have that $r \extends q$.
By absoluteness of $\Pi^1_1$-formulas we have
$$\forces_{\PP} \text{For all } x \in [p] \text{ we have } f^*(x) \in [T]$$
which yields the contradiction
\[
		r \forces \dot{f} = f^*(\dot{e}^{q, \Sigma}(s_\gen)) \in [T].
\]
\end{proof}

\begin{theorem}\label{thm.small.spectra}[CH]
Let $\lambda$ be a cardinal such that $\lambda^{\aleph_0} = \lambda$. Then $\SS^\lambda \forces ``\spec(\aT) = \simpleset{\aleph_1, \lambda}$".
\end{theorem}

\begin{remark}
The set of all finitely splitting trees is an effective Polish space with the topology generated by the basic open sets $\O(T) = \set{S}{S \text{ is a finitely splitting tree with } S \restr n = T}$ where $T$ is a finitely splitting $n$-tree for some $n < \omega$. Hence, we may use that $\SS^\lambda$ has continuous reading of names as above for finitely splitting trees. Furthermore, CH implies that $\SS^\lambda$ has the $\aleph_2$-c.c.
So if we assume that any name for a finitely splitting tree is nice, then we have that its evaluation only depends on $\aleph_1$-many conditions in $\SS^\lambda$.
Here, we already use the notion of a nice name for a finitely splitting tree as defined in the next section.
\end{remark}

\begin{proof}
	First, we prove that $\SS^\lambda \forces \c = \lambda$.
	For every $p \in \SS^\lambda$ let $\dot{F}_p$ be a $\SS^\lambda$-name such that
	\begin{align*}
	\SS^\lambda \forces \dot{F}_p = \{ (f^* &\circ \dot{e}^{p, \Sigma})(s_{\gen}\restr \dom(p)) \mid \Sigma \text{ is a standard enumeration of } \dom(p) \text{ and } \\ &f:{^{<\omega}(^{<\omega}2)} \to \finseqomega \text{ is a code for a continuous function } f^*:{^\omega (^\omega 2)} \to \infseqomega \}
	\end{align*}
	With this definition the continuous reading of names exactly means that
	$$
	\SS^\lambda \forces \infseqomega = \bigcup_{p \in \SS^\lambda} \dot{F}_p
	$$
	By CH we have $\SS^\lambda \forces ``|\dot{F}_p \leq \aleph_1|"$.
	As $\lambda^{\aleph_0} = \lambda$ we have $|\SS^\lambda| = \lambda$, so we compute
	$$
	\SS^\lambda \forces \c = \left| \bigcup_{p \in \SS^\lambda} \dot{F}_p \right| \leq \sum_{p \in \SS^\lambda} \left| \dot{F}_p \right| \leq \sum_{p \in \SS^\lambda} \aleph_1 = |\SS^\lambda| \cdot \aleph_1 = \lambda \cdot \aleph_1 = \lambda
	$$
Conversely, $\SS^\lambda$ adds $\lambda$-many different Sacks reals so we get $\SS^\lambda \forces ``\c \geq \lambda"$. Furthermore, by CH Theorem~\ref{indestructible_adfs} implies the existence of a Sacks-indestructible a.d.f.s.\ family, so $\spec(\aT) \supseteq \simpleset{\aleph_1, \lambda}$. Now, consider any $\aleph_1 < \kappa < \lambda$ and for a contradiction assume $\seq{\dot{T}_\alpha}{\alpha < \kappa}$ are nice $\SS^\lambda$-names for finitely splitting trees and $p \in \SS^\lambda$ is such that
$$p \forces \seq{\dot{T}_\alpha}{\alpha < \kappa} \text{ is a maximal a.d.f.s.\ family}.$$
For any $\alpha < \aleph_2$ choose $p_\alpha \extends p$ with $\dom(p_\alpha) = U_\alpha \in [\lambda]^{\aleph_0}$ such that $\dot{T}_\alpha$ is read continuously below $p_\alpha$. By CH we may use the $\Delta$-system lemma to choose $I_0 \in [\aleph_2]^{\aleph_2}$ and a root $U$ of $\seq{U_\alpha}{\alpha \in I_0}$. By a counting argument there is $I_1 \in [I_0]^{\aleph_2}$ such that $|U_\alpha \setminus U| = |U_\beta \setminus U|$ for all $\alpha, \beta \in I_1$. Then, for every $\alpha \in I_1$, we may choose a bijection $\phi_\alpha : U_\alpha \to \omega$ such that $\phi_\alpha \restr U = \phi_\beta \restr U$ for all $\alpha, \beta \in I_1$. Now, for $\alpha, \beta \in I_1$ define an involution $\pi_{\alpha, \beta}:\lambda \to \lambda$ by
	$$\pi_{\alpha,  \beta}(i) =
	\begin{cases}
	(\phi_\beta)^{-1}(\phi_\alpha(i)) & \text{for } i \in U_\alpha \\
	(\phi_\alpha)^{-1}(\phi_\beta(i)) & \text{for } i \in U_\beta \\
	i & \text{otherwise.}
	\end{cases}
	$$
	Clearly, this is well-defined as $U_\alpha \cap U_\beta = U$ and $\phi_\alpha \restr U = \phi_\beta \restr U$.
	Moreover, for all $\alpha, \beta, \gamma \in I_1$ we have
	\begin{enumerate}
		\item $\pi_{\alpha, \beta}$ maps $U_\alpha$ onto $U_\beta$ and $U_\beta$ onto $U_\alpha$, but is the identity on the rest of $\lambda$.
		\item $\pi_{\alpha,  \beta}$ is the identity on $U$.
		\item $\pi_{\alpha,\alpha} = \text{id}_\lambda$, $\pi_{\alpha,\beta} = \pi_{\beta, \alpha} $ and $\pi_{\alpha, \gamma} = \pi_{\alpha, \beta} \circ \pi_{\beta, \gamma} \circ \pi_{\alpha, \beta}$.
	\end{enumerate}
	$\pi_{\alpha,  \beta}$ naturally induces to an automorphism of $\SS^\lambda$ which we also denote with $\pi_{\alpha,\beta}$.
	Note that the three properties above also hold for the induced maps.
	Fix $\gamma_1 \in I_1$.
	For any $\alpha \in I_1$ we have that $\dom(\pi_{\alpha,  {\gamma_1}}(p_\alpha)) \subseteq U_{\gamma_1}$.
	But by CH $|\SS| = \aleph_1$ and $\aleph_1^{\aleph_0} = \aleph_1$, so there are only $\aleph_1$-many conditions $p \in \SS^\lambda$ with $\dom(p) \subseteq U_{\gamma_1}$.
	Again by a counting argument, we can find $I_2 \in [I_1]^{\aleph_2}$ such that $\pi_{\alpha,  {\gamma_1}}(p_\alpha) = \pi_{\beta,  {\gamma_1}}(p_{\beta})$ for all $\alpha, \beta \in I_2$.
	But this implies that
	$$
	\pi_{\alpha, \beta}(p_\alpha) = \pi_{\alpha, {\gamma_1}} \circ \pi_{{\gamma_1}, \beta} \circ \pi_{\alpha, {\gamma_1}}(p_\alpha) = \pi_{\alpha, {\gamma_1}} \circ \pi_{{\gamma_1}, \beta} \circ \pi_{\beta, {\gamma_1}}(p_\beta) = \pi_{\alpha, {\gamma_1}}(p_\beta) = p_\beta
	$$
	for all $\alpha, \beta \in I_2$.
	Notice that the last equality follows from $\dom(p_\beta) \cap (U_\alpha \cup U_{\gamma_1}) = U$.

Again, fix some $\gamma_2 \in I_2$. The mapping  $\pi_{\alpha,  \beta}$ induces an automorphism of $\SS^\lambda$-names and note again that the three properties above also hold for this extension. By the automorphism theorem we get that $\pi_{\alpha, {\gamma_2}}(\dot{T}_\alpha)$ can be read continuously below $\pi_{\alpha, {\gamma_2}}(p_\alpha) = p_{\gamma_2}$.
But by CH there are only $\aleph_1$ many different names for finitely splitting trees that can be read continuously below $p_{\gamma_2}$. Again by a counting argument we can find $I_3 \in [I_2 \setminus \{\gamma_2\}]^{\aleph_2}$ such that for all $\alpha, \beta \in I_3$ we have
	$$
	p_{\gamma_2} \forces \pi_{\alpha,  {\gamma_2}}(\dot{T}_\alpha) = \pi_{\beta,  {\gamma_2}}(\dot{T}_\beta).
	$$
We may assume that for all $\alpha < \kappa$, all maximal antichains in the nice name of $\dot{T}_\alpha$ restrict to maximal antichains below $p_\alpha$. Then, we write $\dot{T}_\alpha \restr p_\alpha$ for the restriction of the name $\dot{T}_\alpha$ below $p_\alpha$.
For all $\alpha \in I_3$, $\pi_{\alpha, \gamma_2}(\dot{T}_{\alpha} \restr p_\alpha)$ can be read continuously below $p_{\gamma_2}$, so we may assume that the maximal antichains have been chosen such that $\dom(\pi_{\alpha, \gamma_2}(\dot{T}_{\alpha} \restr p_\alpha)) = \dom(p_{\gamma_2}) = U_{\gamma_2}$.
	
We now define a new $\SS^\lambda$-name $\dot{T}_\kappa$ for a finitely splitting tree. By assumption, all names $\dot{T}_\alpha$ are nice, so we may choose $\set{p_{\alpha, i}}{i < \aleph_1, \alpha < \kappa}$ such that the evaluation of $\dot{T}_\alpha$ only depends on $\set{p_{\alpha,i}}{i < \aleph_1}$. Let $W_\alpha = \bigcup_{i < \aleph_1} \dom(p_{\alpha, i})$ and $W = \bigcup_{\alpha < \kappa} W_\alpha$. Then $|W| \leq \kappa < \lambda$, so we may choose $U_\kappa \in [\lambda]^{\aleph_0}$ such that $U_\kappa \cap W = U$ and $|U_\gamma \setminus U| = |U_\kappa \setminus U|$. Choose a bijection $\phi_\kappa: U_\kappa \to \omega$ such that $\phi_\kappa \restr U = \phi_\alpha \restr U$ for all $\alpha \in I_1$. Thus, we can extend the system of involutions to $I_1 \cup \simpleset{\kappa}$ by defining $\pi_{\alpha, \kappa}$ by the same equation as above, so that the three properties are preserved. Fix $\gamma_3 \in I_3$ and set $p_\kappa = \pi_{\gamma_3, \kappa}(p_{\gamma_3})$ and $\dot{T}_\kappa = \pi_{{\gamma_3}, \kappa}(\dot{T}_{\gamma_3})$. We claim that $p_\kappa$ and $\dot{T}_\kappa$ are independent of the choice of $\gamma_3$ in the following sense: For all $\alpha \in I_2$ (in particular for $\gamma_2$) we have $\pi_{\alpha, \kappa}(p_\alpha) = p_\kappa$ and for all $\alpha \in I_3$ we have 
$$p_\kappa \forces \pi_{\alpha, \kappa}(\dot{T}_\alpha) = \dot{T}_\kappa = \pi_{\gamma_3, \kappa}(\dot{T}_{\gamma_3}).$$
 For $\alpha \in I_2$ the first claim holds because
	$$
	\pi_{\alpha, \kappa}(p_\alpha) = \pi_{\alpha, \gamma_3} \circ \pi_{\gamma_3, \kappa} \circ \pi_{\alpha, \gamma_3}(p_\alpha) = \pi_{\alpha, \gamma_3} \circ \pi_{\gamma_3, \kappa} (p_{\gamma_3}) = \pi_{\alpha, \gamma_3}(p_{\kappa}) = p_\kappa
	$$
	where the last equality follows from $\dom(p_\kappa) \cap (U_\alpha \cup U_{\gamma_3}) = U$.
For the second claim let $\alpha \in I_3$ and apply the automorphism theorem to
	$$
	p_{\gamma_2} \forces \pi_{\alpha, \gamma_2}(\dot{T}_\alpha) = \pi_{\gamma_3, \gamma_2}(\dot{T}_{\gamma_3})
	$$
	to obtain
	$$
	p_\kappa \forces \pi_{\gamma_2, \kappa}(\pi_{\alpha, \gamma_2}(\dot{T}_\alpha)) = \pi_{\gamma_2, \kappa}(\pi_{\gamma_3, \gamma_2}(\dot{T}_{\gamma_3})).
	$$
Using the third property of the involutions and simplifying yields
$$p_\kappa \forces \pi_{\alpha, \gamma_2}(\pi_{\alpha, \kappa}(\dot{T}_\alpha)) = \pi_{\gamma_3, \gamma_2}(\pi_{{\gamma_3}, \kappa}(\dot{T}_{\gamma_3})).$$
To prove the claim it remains to show that
$$p_\kappa \forces \pi_{\alpha, \kappa}(\dot{T}_\alpha) = \pi_{\alpha, \gamma_2}(\pi_{\alpha, \kappa}(\dot{T}_\alpha)) \text{ and } \pi_{{\gamma_3}, \kappa}(\dot{T}_{\gamma_3}) = \pi_{\gamma_3, \gamma_2}(\pi_{{\gamma_3}, \kappa}(\dot{T}_{\gamma_3})).$$
First, since $\alpha \in I_3$ we have $U_\alpha \cap U_{\gamma_2} = U$. Remember $\pi_{\alpha, \gamma_2}(\dot{T}_\alpha \restr p_\alpha)$ only depends on $U_{\gamma_2}$, so it does not depend on $U_\alpha \setminus U$. But then $\pi_{\gamma_2, \alpha}(\pi_{\alpha, \gamma_2}(\dot{T}_\alpha \restr p_\alpha)) = \dot{T}_\alpha \restr p_\alpha$ does not depend on $U_{\gamma_2} \setminus U$ and clearly also not on $U_\kappa \setminus U$. Hence, $\pi_{\alpha, \kappa}(\dot{T}_\alpha \restr p_\alpha)$ does not depend on $U_{\gamma_2} \setminus U$ and $U_{\alpha} \setminus U$.
This shows that $$p_\kappa \forces \pi_{\alpha, \gamma_2}(\pi_{\alpha, \kappa}(\dot{T}_\alpha)) = \pi_{\alpha, \gamma_2}(\pi_{\alpha, \kappa}(\dot{T}_\alpha \restr p_\alpha)) = \pi_{\alpha, \kappa}(\dot{T}_\alpha \restr p_\alpha) = \pi_{\alpha, \kappa}(\dot{T}_\alpha).$$
Notice, that the second equation follows analogously using $\gamma_3 \in I_3$.
	
Finally, let $\beta < \kappa$. Choose $\alpha \in I_3$ such that $U_\alpha \cap W_\beta \subseteq U$. This is possible, as we have $|W_\beta| \leq \aleph_1 < |I_3| = \aleph_2$ and for every $i \in W_\beta \setminus U$ there is at most one $\alpha \in I_3$ such that $i \in U_\alpha$ as $\seq{U_\alpha}{\alpha \in I_3}$ is a $\Delta$-system. But then $\pi_{\alpha, \kappa}(\dot{T}_\beta) = \dot{T}_\beta$ and $p_\kappa \forces `` \dot{T}_\kappa = \pi_{\alpha, \kappa}(\dot{T}_\alpha)"$. Now, applying the automorphism $\pi_{\alpha, \kappa}$ to
$$p_{\alpha} \forces \dot{T}_\alpha \text{ and } \dot{T}_\beta \text{ are almost disjoint} $$
yields that
$$p_\kappa \forces \pi_{\alpha, \kappa}(\dot{T}_\alpha) = \dot{T}_\kappa \text{ and } \dot{T}_\beta \text{ are almost disjoint.}$$
But $\dom(p) \subseteq U$ implies $p_\kappa \extends p$, which contradicts that
$$	p \forces \seq{\dot{T}_\alpha}{\alpha < \kappa} \text{ is a maximal a.d.f.s.\ family.}
$$
\end{proof}

Finally, we discuss how our results answer Question 2 of Newelski's \cite{New87}.
In fact, if we replace the use of Lemma \ref{extendctbladfs} in the proof of Theorem \ref{indestructible_adfs} by the following lemma

\begin{lemma}\label{extendctblnwd}
	Let $\T$ be a countable family of almost disjoint nowhere dense trees on $^\omega 2$, $\lambda$ be a cardinal, $p \in \SS^\lambda$ and $\dot{f}$ be a $\SS^\lambda$-name for an element of $^\omega 2$ such that for all $T \in \T$ we have
	$$p \forces \dot{f} \notin [T].$$
	Then there is a nowhere dense tree $S$ and $q \extends p$ such that $\T \cup \simpleset{S}$ is almost disjoint and
	$$q \forces \dot{f} \in [S].$$
\end{lemma}
then we may obtain the following result with exactly the same proof:

\begin{theorem}[CH]
	There is a partition of $^\omega 2$ into nowhere dense closed sets which is indestructible by forcing with countably supported iteration or product of Sacks forcing of arbitrary length.
\end{theorem}

Notice that also Lemma \ref{extendctblnwd} has a similar proof as Lemma \ref{extendctbladfs}, where a routine fusion argument can be used to force the hull for the name of an element of $^\omega 2$ to be a nowhere dense tree on $^\omega 2$.
Finally, with a bit of extra work it can be shown that a part of the proof of Theorem \ref{indestructible_adfs} in fact implies that every partition indestructible by countable product of Sacks-forcing is already indestructible by any countably supported product or iteration of Sacks-forcing. This shows that also the partition forced by Newelski in \cite{New87} is indestructible by countably supported iterations of Sacks-forcing.
A complete argument is contained in the author's follow-up paper \cite{fischerschembecker}.

\section{The consistency of $\d < \aT = \a$}

In all our models that we considered so far we have that $\d = \aT$.
In this section we prove that consistently $\mu=\d < \aT = \lambda$ can hold for various values above a measurable cardinal.
In \cite{shelahtemplate} Shelah introduced template iterations to prove the consistency of $\d < \a$; see also \cite{brendletemplate} for an insightful exposition of Shelah's results.
Shelah first proved that one can iterate templates canonically by taking ultrapowers and use an ``average of names"-argument to obtain the consistency of $\kappa < \d < \a = \c$ above a measurable cardinal $\kappa$.
We will prove that in this model we also have that $\aT = \c$ by a similar ``average of names"-argument.
Notably, Shelah constructed a second template iteration, so that a more involved ``isomorphism of names"-argument yields a model of $\aleph_2 \leq \d < \aT$ without the use of a measurable cardinal.
Even though these results finally solved the consistency of $\d < \a$ these technique cannot be used to settle the consistency of $\d = \aleph_1 < \a = \aleph_2$.

Throughout this section fix a measurable cardinal $\kappa$ and a $<\!\!\kappa$-complete ultrafilter $\U$ on $\kappa$.
First, we briefly restate the basic definitions and properties of the ultrapower forcing given in \cite{shelahtemplate}:

\begin{definition}
Let $\PP$ be a forcing. The ultrapower forcing of $\PP$ by $\U$ is defined as the set of all equivalence classes
$$\PP^\kappa/\U = \set{[f]_\U}{f:\kappa \to \PP}$$
where $[f]_\U = [g]_U$ iff $\set{\alpha < \kappa}{f(\alpha) = g(\alpha)} \in \U$. Usually, we will drop the subscript $\U$.
Furthermore, we order $\PP^\kappa / \U$ by $[f]_\U \leq [g]_\U$ iff $\set{\alpha < \kappa}{f(\alpha) \leq g(\kappa)} \in \U$.
It is easy to see that this defines a partial order. Furthermore, we have an embedding $\PP \to \PP^\kappa/\U$ of partial orders, where $p \in \PP$ is mapped to the equivalence class of the constant map $f_p$, i.e.\ $f_p(\alpha) = p$ for all $\alpha \in \kappa$. Hence, we may identify $p$ with $[f_p]$.
\end{definition}

\begin{lemma}
Let $\PP$ be a forcing. Then the following statements hold:
	\begin{enumerate}
		\item $\PP \completesubposet \PP^\kappa/\U$ iff $\PP$ is $\kappa$-cc. 
		\item If $\mu < \kappa$ and $\PP$ is $\mu$-cc, then also $\PP^\kappa/\U$ is $\mu$-cc.
	\end{enumerate}
\end{lemma}

\begin{proof}
	See \cite{brendletemplate}.
\end{proof}

From now on fix a c.c.c.\ forcing $\PP$, so that both items of the previous lemma apply. Analogous to the average of $\PP$-names of reals in \cite{brendletemplate} we consider the average of $\PP$-names of finitely splitting trees. Let $\A(\PP)$ be the set of all maximal antichains in $\PP$.

\begin{definition}
The pair $(A, T_\bullet)$, where $T_\bullet=\langle T_n: n\in\omega\rangle$, is a nice $\PP$-name for a finitely splitting tree iff $A:\omega \to \A(\PP)$ and for every $n < \omega$ we have that $T_n:A(n) \to \P(\finseqleqomegaarg{n})$ such that
\begin{enumerate}
		\item For all $n < \omega$ and $p \in A(n)$ we have that $T_n(p)$ is a finitely splitting $n$-tree.
		\item For all $n < m$ and $p\in A(n), q \in A(m)$ with $p \compat q$ we have that $T_m(q) \cap \finseqleqomegaarg{n} = T_n(p)$.
\end{enumerate}
\end{definition}

Given any $\PP$-name $\dot{T}$ for a finitely splitting tree, for every $n < \omega$ we may choose an antichain $A(n)$ such that every element $p \in A(n)$ decides $\dot{T} \restr n$ as $T_n(p)$.
It is easy to see that then the second item is also satisfied by these choices, so we have defined a nice $\PP$-name $(A, T_\bullet)$ for a finitely splitting tree.
But from $(A, T_\bullet)$ we can define a $\PP$-name $\dot{S}$ for a finitely splitting tree by
$$
\dot{S} = \set{(\check{s},p)}{\text{there is an } n < \omega \text{ such that } p \in A(n) \text{ and } s \in T_n(p)}
$$
It is easy to see that $\PP \forces ``\dot{T} = \dot{S}"$, so we in fact showed that $\dot{T}$ can be represented by a nice $\PP$-name for a finitely splitting tree.
Hence, in the following we may only consider nice names for finitely splitting trees.
Using this, we want to understand how nice $\PP^\kappa/\U$-names for finitely splitting trees can be constructed from nice $\PP$-names for finitely splitting trees and vice versa.

For $\alpha < \kappa$ let $(A^\alpha, T^\alpha_\bullet)$ be nice $\PP$-names for finitely splitting trees. Fix $\alpha < \kappa$ and for $n < \omega$ enumerate $A^\alpha(n) = \set{p^\alpha_{n,i}}{i < \omega}$. For fixed $n,i < \omega$ consider $[p_{n,i}] = \seq{p^{\alpha}_{n,i}}{\alpha < \kappa} / \U \in \PP^{\kappa} / \U$ and $T_{n,i} = [T_{n,i}] = \seq{T^\alpha_n(p^\alpha_{n,i})}{\alpha < \kappa} / \U \in (\P(\finseqleqomegaarg{n}))^\kappa / \U = \P(\finseqleqomegaarg{n})$. By countable completeness of $\U$ we get that  $A(n) = \set{[p_{n,i}]}{i < \omega}$ is a maximal antichain for all $n < \omega$. Set $T_n([p_{n,i}]) = T_{n,i}$. We claim that $(A,T_\bullet)$ is a nice $\PP^\kappa / \U$-name for a finitely splitting tree, which we call the average of $(A^\alpha, T^\alpha_\bullet)$.
	
But $(1)$ follows from the fact that $T^\alpha_n(p^\alpha_{n,i})$ is a finitely splitting tree for all $\alpha < \kappa$ and $n,i < \omega$. For item $(2)$ let $n < m$ and assume $[p_{n,i}] \compat [p_{m,j}]$ for some $i,j < \omega$. But then
$$\set{\alpha < \kappa}{p^\alpha_{n,i} \compat p^\alpha_{m,j}} \cap \set{\alpha < \kappa}{T_{n,i} = T^\alpha_n(p^\alpha_{n,i})} \cap \set{\alpha < \kappa}{T_{m,j} = T^\alpha_m(p^\alpha_{m,j})} \in \U$$
so choose such an $\alpha < \kappa$. Then we have that
$$T_m([p_{m,j}]) = T_{m,j} = T^\alpha_m(p_{m,j}^\alpha) \text{ and } T_n([p_{n,i}]) = T_{n,i} = T^\alpha_n(p_{n,i}^\alpha)$$
which implies $T_m([p_{m,j}]) \cap \finseqleqomegaarg{n} = T^\alpha_m(p_{m,j}^\alpha) \cap \finseqleqomegaarg{n} = T^\alpha_n(p_{n,i}^\alpha) = T_n([p_{n,i}])$, since $p^\alpha_{n,i} \compat p^\alpha_{m,j}$.
	
Conversely, assume we have a nice $\PP^\kappa / \U$-name $(A,T_\bullet)$ for a finitely splitting tree. For $n < \omega$ enumerate $A(n) = \set{[p_{n,i}]}{i < \omega}$. Note that by countable completeness of $\U$
$$D = \set{\alpha < \kappa}{\set{p^\alpha_{n,i}}{i < \omega} \text{ is a maximal antichain for all } n < \omega} \in \U.$$
Thus, by modifying the $p^\alpha_{n,i}$ on a small set with respect to the ultrafilter $\U$ we may also assume $A^\alpha(n) = \set{p^{\alpha}_{n,i}}{i < \omega}$ is a maximal antichain for all $n < \omega$ and $\alpha < \kappa$. But then, defining $T^\alpha_{n}(p^\alpha_{n,i}) = T_n([p_{n,i}])$ yields a nice $\PP$-name $(A^\alpha, T^\alpha_\bullet)$ for a finitely splitting tree for all $\alpha < \kappa$.

\begin{lemma}
	Assume $\dot{\T}$ is a $\PP$-name for an a.d.f.s.\ family of size $\lambda \geq \kappa$.
	Then
$$\forces_{\PP^\kappa / \U} \dot{\T} \text{ is not a maximal a.d.f.s.\ family.}$$
\end{lemma}

\begin{proof}
	Choose $\PP$-names for finitely splitting trees $\dot{T}^\alpha$ such that
	$$
	\forces_\PP \dot{\T} = \set{\dot{T}^\alpha}{\alpha < \lambda}
	$$
	For $\alpha < \kappa$ choose nice names $(A^\alpha, T^\alpha_\bullet)$ for $\dot{T}^\alpha$ and enumerate $A^\alpha(n) = \set{p^\alpha_{n,i}}{i < \omega}$.
	The average $(A, T_\bullet)$ of $\seq{(A^\alpha, T^\alpha_\bullet)}{\alpha < \kappa}$ defined as in the previous remark is a nice $\PP^\kappa / \U$-name for a finitely splitting tree.
	Let $\dot{T}$ be the $\PP^\kappa / \U$-name for $(A, T_\bullet)$.
	We claim that
	$$
	\forces_{\PP^\kappa / \U} \dot{T} \text{ is almost disjoint from } \dot{\T}
	$$
	so let $\beta < \lambda$.
	For all $\alpha < \kappa$ with $\alpha \neq \beta$ there is a maximal antichain $\set{q^\alpha_i}{i < \omega}$ and $\set{n^\alpha_i}{i < \omega}$ such that
	$$
	q^\alpha_i \forces_\PP \dot{T}^\alpha \cap \dot{T}^\beta \subseteq \finseqleqomegaarg{n^\alpha_i}
	$$
	Consider $[q_i] = \seq{q^\alpha_i}{\alpha < \kappa} / \U \in \PP^{\kappa} / \U$ and $n_i = [n_i] = \seq{n^\alpha_i}{\alpha < \kappa} / \U \in \omega^\kappa / \U = \omega$.
	Again, by countable completeness of $\U$ we have that $\set{[q_i]}{i < \omega}$ is a maximal antichain.
	We prove that for every $i < \omega$ we have
	$$
	[q_i] \forces_{\PP^\kappa / \U} \dot{T}^\beta \cap \dot{T} \subseteq \finseqleqomegaarg{n_i}
	$$
	Assume not.
	Then there are $i < \omega$, $l > n_i$, $s \in \finseqleqomegaarg{l} \setminus \finseqleqomegaarg{n_i}$ and $[r] \leq [q_i]$ such that
	$$
	[r] \forces_{\PP^\kappa / \U} s \in \dot{T}^\beta \cap \dot{T}
	$$
	By possibly extending $[r]$ we may assume that there is a $j < \omega$ such that $[r] \extends [p_{l,j}]$ and also $s \in T_{l}([p_{l,j}])$.
	Furthermore, let $[r] = \seq{r^\alpha}{\alpha < \kappa} / \U$.
	Then we have
	$$
	\set{\alpha < \kappa}{r^\alpha \extends p^\alpha_{l,j}, r^\alpha \extends q^\alpha_i, T^\alpha_{l}(p^\alpha_{l,j}) = T_{l}([p_{l,j}]), n_i^\alpha = n_i \text{ and } r^\alpha \forces_\PP s \in \dot{T}^\beta} \in \U 
	$$
	so choose such an $\alpha < \kappa$. 
	Then since $r^\alpha \extends p^\alpha_{l,j}$, $T^\alpha_{l}(p^\alpha_{l,j}) = T_{l}([p_{l,j}])$ and $s \in T_{l}([p_{l,j}])$ we have that
	$$
	r^\alpha \forces_\PP s \in \dot{T}^\alpha \cap \dot{T}^\beta
	$$
	but also
	$$
	q^\alpha_i \forces_\PP \dot{T}^\alpha \cap \dot{T}^\beta \subseteq \finseqeqomegaarg{n^\alpha_i}
	$$
	which is a contradiction, since $r_\alpha \extends q^\alpha_i$ and $s \in \finseqleqomegaarg{l} \setminus \finseqleqomegaarg{n_i}$, but $n_i = n^\alpha_i$.
\end{proof}

Thus, we can extend the result in \cite{shelahtemplate} to $\aT$ as well:
\begin{theorem}\label{thm.aT.ultrapowers}
Assume $\kappa$ is measurable and $\kappa < \mu < \lambda$, $\lambda = \lambda^\omega$ are regular cardinals such that  $\nu^\kappa < \lambda$ for all $\nu < \lambda$. Then there is a forcing extension satisfying $\b = \d = \mu$ and $\a= \aT = \c = \lambda$.
\end{theorem}

\begin{proof}
	This holds in Shelah's template iteration for iterating ultrapowers.
	There are no maximal a.d.f.s.\ families of size $<\!\!\mu$ as $\d \leq \aT$.
	Furthermore, the proof that there are no maximal a.d.f.s.\ families of size $\xi \in [\mu,\lambda)$ works completely analogous as the proof of Theorem 2.3 in \cite{brendletemplate} where the use of Lemma 0.3 is replaced with the previous lemma.
\end{proof}

\bibliographystyle{plain}
\bibliography{refs}

\end{document}